\documentclass[12pt]{amsart}

\usepackage{amscd} 
\usepackage{mathrsfs}

\newtheorem{theorem}{Theorem}[section]
\newtheorem{lemma}[theorem]{Lemma}
 
\theoremstyle{definition}

\numberwithin{equation}{section}

\def\z*{\bar z}

\def\B{{\mathscr B}}

\def\K{\mathsf K}

\def\X{\mathsf X}
\def\Y{\mathsf Y}

\def\C{\mathcal C}

\def\dom{\text{\rm dom}}
\def\ran{\text{\rm ran}}
\def\supp{\text{\rm supp}}

\def\RE{\mathbb R}
\def\CO{{\mathbb C}}

\def\SL{S\! L}
\def\DL{D\! L}

\def\ph*{\phi_\Diamond}

\def\be{\begin{equation}}
\def\ee{\end{equation}}
\def\min{{\rm min}}

\def\-{{\rm in}}
\def\+{{\rm ex}}

\def\bou{{\mathscr B}}
\def\Cos{\mbox{Cos}}
\def\Sin{\mbox{Sin}}
\begin{document}

\title[Inverse wave scattering in the Laplace domain]
{Inverse wave scattering in the Laplace domain:
a factorization method approach}
\author{Andrea Mantile}
\author{Andrea Posilicano}

\address{Laboratoire de Math\'{e}matiques, Universit\'{e} de Reims -
FR3399 CNRS, Moulin de la Housse BP 1039, 51687 Reims, France}
\address{DiSAT, Sezione di Matematica, Universit\`a dell'Insubria, via Valleggio 11, I-22100
Como, Italy}
\email{andrea.mantile@univ-reims.fr}
\email{andrea.posilicano@uninsubria.it}


\begin{abstract} Let $\Delta_{\Lambda}\le \lambda_{\Lambda}$ be a semi-bounded self-adjoint realization of the Laplace operator with boundary conditions (Dirichlet, Neumann, semi-transparent) assigned on the Lipschitz boundary of a bounded obstacle $\Omega$. Let $u^{\Lambda}_{f}$ and $u^{0}_{f}$ denote the solutions of the wave equations corresponding  to $\Delta_{\Lambda}$ and to the free Laplacian $\Delta$ respectively, with a source term $f$ concentrated at time $t=0$ (a pulse). We show that for any fixed $\lambda>\lambda_{\Lambda}\ge 0$ and any fixed $B\subset\subset{\mathbb R}^{n}\backslash\overline\Omega$, the obstacle $\Omega$ can be reconstructed by the data 
$$
F^{\Lambda}_{\lambda}f(x):=\int_{0}^{\infty}e^{-\sqrt\lambda\,t}\big(u^{\Lambda}_{f}(t,x)-u^{0}_{f}(t,x)\big)\,dt\,,\qquad x\in B\,,\  f\in L^{2}({\mathbb R}^{n})\,,\ \mbox{supp}(f)\subset B\,.
$$
A similar result holds in the case of screens reconstruction, when the boundary conditions are assigned only on a part of the boundary. Our method exploits the factorized form of the resolvent difference 
$(-\Delta_{\Lambda}+\lambda)^{-1}-(-\Delta+\lambda)^{-1}$. 
\end{abstract}
\maketitle
\section{Introduction.}

We consider the problem of obstacles' reconstruction from measurements of
time-dependent scattered waves. Different approaches have been developed; in 
\cite{LuPo} time-Fourier transform is used to process data in the frequency
domain via the point source method. In \cite{ChHaLeMo} the case of Dirichlet
obstacles is considered in the time-dependent setting by using measurements
of causal waves, that is, waves such that $u(t,x)=0$ for $t\le T$ (see also 
\cite{HaLeMa} for Neumann and Robin obstacles). In these works a linear
sampling method is adapted to work on time-domain data without using the
Fourier transform. As remarked in \cite{ChHaLeMo}, this approach may offer a
better quality of the reconstruction compared to frequency domain methods
working with a single frequency; nevertheless, an approximation argument
prevents a rigorous mathematical characterization of the obstacle. A
different strategy, based on an adaptation of the factorization method to
the time-domain, has been recently proposed in \cite{CaHaLe}. There, the
authors introduce a far field operator for Dirichlet obstacles scattering in
the time dependent setting; the inverse data are given by measurements of
scattered causal waves in the far field regime. However, to obtain a symmetric factorization containing a
coercive middle operator (needed to implement the factorization method), a
perturbed far field operator, arising from artificially modified
measurements in the time range $(T,+\infty)$, is introduced. The analytical tool to study this
factorization is provided by Laplace transform
analysis of retarded potentials. Here we exploit a different approach in which the far field 
operator is directly defined by the Laplace transform of the scattered field, so to take advantage of the factorized form coming out of the resolvent difference between the generators of the free and the interacting dynamics; this avoids the introduction of any kind of perturbation (unphysical or not) of the far field operator, the latter  being already factorized in terms of a coercive middle operator. Our approach has some resemblance with a method using the Laplace transform of Maxwell's equations to obtain a time-independent and coercive system which has been used in the PhD thesis \cite{Jais} (we thank an anonymous referee for pointing out to us such a reference).
\par
In our recent paper \cite{MP-inv}, we used Kre\u{\i}n-type resolvent
formulae, combined with the limiting absorption principle and the
factorization method, to provide inverse scattering (reconstrucion) results
for Lipschitz obstacles and screens. In what follows analogous results are
obtained for the time-dependent obstacle scattering problem; our approach
exploits the Laplace transform analysis of the time-propagator leading to a
factorized representation of the data operator in the Laplace transform
domain. Then, using sampling methods, obstacles and screens can be
reconstructed by the knowledge, for some fixed $\lambda >0$ and $B\subset
\subset \mathbb{R}^{n}\backslash \overline{\Omega }$, of the data 
\begin{equation*}
F_{\lambda }^{\Lambda }f(x):=\int_{0}^{\infty }e^{-\sqrt{\lambda }\,t}\big(%
u_{f}^{\Lambda }(t,x)-u_{f}^{0}(t,x)\big)\,dt\,,\qquad x\in B\,,\ f\in L^{2}(%
\mathbb{R}^{n})\,,\ \mbox{supp}(f)\subset B\,,
\end{equation*}%
where $u_{f}^{\Lambda }$ and $u_{f}^{0}$ solve the inhomogeneous wave
equations for $\Delta _{\Lambda }$ and $\Delta $ respectively with a pulse $f
$ concentrated at time $t=0$ (see Theorems \ref{thm} and \ref{thm-s} for the
precise statements). Here $\Delta $ is the self-adjoint Laplacian in the
whole space, describing free waves propagation, while $\Delta _{\Lambda }$
is the self-adjoint realization of the Laplacian with boundary conditions on
(part of) the boundary $\Gamma $ of the obstacle $\Omega $ which are
uniquely identified by the choice of the operator $\Lambda $ acting on
functions on $\Gamma $. The allowed $\Lambda $'s permit to consider many
local boundary conditions, in particular Dirichlet, Neumann and
semi-transparent ones, either assigned on the whole boundary $\Gamma $ (see
Section \ref{closed}) or on a relatively open piece $\Sigma $ (see Section %
\ref{unclosed}).
\par
Our modeling is inspired by an ideal experimental setup where the incident
wave is generated by pulses space-localized in a fixed open and bounded
region $B$ and the measurements are performed by detectors placed in the
same domain $B$ (see for instance \cite{borcea}). This choice allows to use
in our computations results from wave propagation theory under standard
regularity assumptions. In the applications perspective, a more appropriate
choice would consist in replacing $B$ with the boundary of an open domain.
This setting would be closer to the one proposed in \cite{ChHaLeMo} and \cite%
{CaHaLe}, where the incident fields are generated by a density on a sphere,
while (concerning \cite{CaHaLe}) the data operator output, i.e. the physical
measure, depicts the behavior of scattered fields at far distances and
large times. The same kind of arguments used here still apply to this case
and we do not expect relevant changes in our conclusions in such a modified
setting. Nevertheless, assigning pulse-sources and measurements on a
boundary would require the study of the mapping properties in distributional spaces of the resolvents appearing in our Kre\u\i n's type formulae; this case will be considered in a successive paper.
\par
A relevant feature of our approach rests upon the fact that it avoids
unphysical modifications of the data and provides a rigorous reconstruction
algorithm; let us point out that our results, since the algorithm is exact, immediately imply uniqueness: the obstacle is univocally determined by the scattered field (see, e.g., the review paper \cite{Isakov} for the many facets of the uniqueness problem in inverse scattering). As in the aforementioned works, we require
global-in-time data; this is due to the use of the Laplace transform. The
error introduced by using finite-time data is considered: we provide an
estimate regarding the difference (in uniform operator norm) between the
experimentally realistic operator $F_{\lambda}^{\Lambda,\varepsilon,t_{
\circ}}$ defined in term of pulses concentrated on small time intervals $
[0,\varepsilon]$, $\varepsilon\ll1$, and measurements lasting a finite time $
t_{\circ}\gg\varepsilon$, and the ideal operator $F_{\lambda}^{\Lambda}$
corresponding to the limit case $t_{\circ }=+\infty$, $\varepsilon=0$. In
Lemma \ref{estimate} we show that it is of order $(e^{-\lambda
t_{\circ}}+\varepsilon )\lambda^{-1/2}$; thus it can be made arbitrarily
small by taking $\lambda$ sufficiently large. However, let us point out that this does not  mean neither that the factorization method will work for the operator $F_{\lambda}^{\Lambda,\varepsilon,t_{\circ}}$ nor that the error introduced by replacing $F_{\lambda}^{\Lambda}$ with $F_{\lambda}^{\Lambda,\varepsilon,t_{\circ}}$ in formulae \eqref{inf}, \eqref{1/4} and \eqref{inf-s}, \eqref{1/4-s} will be necessarily small; the stability problem in inverse scattering problems is notoriously difficult and in this short communication we do not treat this issue. 
\vskip10pt\noindent
{\bf Acknowledgements.} The authors are indebted to Guanghui Hu for the 
fruitful discussions which largely inspired this work.
\section{The resolvent formula for Laplacians with boundary conditions.}
Let $\Delta:H^{2}(\RE^{n})\subseteq L^{2}(\RE^{n})\to L^{2}(\RE^{n})$ be the self-adjoint operator given by the free Laplacian on the whole space; here $H^{2}(\RE^{n})$ denotes the usual Sobolev space  of square integrable functions with square integrable distributional Laplacian. Another  self-adjoint operator $\widetilde\Delta:\dom(\widetilde\Delta)\subseteq L^{2}(\RE^{n})\to L^{2}(\RE^{n})$ is said to
be a singular perturbation of $\Delta$ if the set 
\be\label{SP}
\{u\in H^{2}(\RE^{n})\cap \dom(\widetilde\Delta)\, :\, \Delta u=\widetilde\Delta u\}
\ee 
is dense in
$L^{2}(\RE^{n})$.  In concrete situations $\widetilde\Delta$ represents the Laplace operator with some kind of boundary conditions on a null subset. We notice that  $\widetilde\Delta$ is a self-adjoint extension of the symmetric operator $\Delta^{\circ}$ given by restricting the free Laplacian to the set defined in \eqref{SP}. Therefore the singular perturbations of the Laplacian can be realized as self-adjoint extensions of the symmetric operators given by the restrictions of the Laplacian $\Delta$ to  subspaces which are dense in $L^{2}(\RE^{n})$ and closed in $H^{2}(\RE^{n})$. Without loss of generality we can suppose that such subspaces are the kernels of some bounded linear maps. This leads to introduce the following framework.\par
Given an auxiliary Hilbert space $\K$, we introduce a linear application $\tau:H^{2}(\RE^{n})\to\K$
which plays the role of an abstract trace (evaluation) map. We assume that  \vskip5pt\noindent 
1. $\tau$ is continuous; \par\noindent
2. $\tau$ is surjective (so that $\K$ plays the role of the trace space);\par\noindent 
3. ker$(\tau)$ is dense in $L^{2}(\RE^{n})$. 
\vskip5pt\noindent In the following we do not identify $\K$ with its dual $\K^{*}$; however we use $\K^{**}\equiv\K$. We suppose that there exists a Hilbert space $\K_{0}$ and continuous embeddings with dense range $\K\hookrightarrow\K_{0}\hookrightarrow\K^{*}$; then the $\K$-$\K^{*}$ duality $\langle\cdot,\cdot\rangle_{\K^{*}\!,\K}$ (conjugate-linear with respect to the first variable) is defined in terms of the scalar product of  $\K_{0}$. \par For any $z$ in the resolvent set $\CO\backslash(-\infty,0]$,  we define the bounded operators $$R^{0}_{z}:=(-\Delta+z)^{-1}:L^{2}(\RE^{n})\to H^{2}(\RE^{n})$$  and 
\be\label{Gz}
G_{z}:=(\tau R^{0}_{\bar z})^*:\K^{*}\to L^{2}(\RE^{n})\,.
\ee
Given a couple of reflexive Banach spaces $\X$, $\Y$, with $\K\hookrightarrow\X$ (and hence $\X^{*}\hookrightarrow\K^{*}$), and given the bounded operator $P\in\bou(\X,\Y)$, we consider a family of maps $M_{z}\in \B(\X^{*},\X)$, $z\in \CO\backslash(-\infty,0]$, such that 
\be\label{MM}
M_{z}^{*}=M_{\bar z}\,,\qquad M_{z}-M_{w}=(z-w)\,G^{*}_{\bar w}G_{z}
\ee
and then, supposing that 
\be
Z_{M}:=\{z\in \CO\backslash(-\infty,0]: (PM_{z}P^{*})^{-1}\in\bou(\Y,\Y^{*}),\,(PM_{\bar z}P^{*})^{-1}\in\bou(\Y,\Y^{*})\}\,,
\ee
is not empty, we define the  map 
$$
\Lambda:Z_{M}\to\B(\X ,\X^*)\,,\qquad z\mapsto\Lambda_{z}:=P^{*}(PM_{z}P^{*})^{-1}P\,.
$$
By \eqref{MM} one gets (see \cite[relation (2) and (4)]{P01})
\be\label{Lambda}
\Lambda_{z}^{*}=\Lambda_{\bar z}\,,
\qquad
\Lambda_{w}-\Lambda_{z}=(z-w)\Lambda_{w}G_{\bar w}^{*}G_{z}\Lambda_{z}
\ee
and so, by \cite[Theorem 2.1]{P01}, one gets the following result (see \cite[Theorem 2.4]{JMPA}; here we also take into account \cite[Theorem 2.19]{CFP}). 
\begin{theorem}\label{WO} Let $\tau$, $P$, $M$  be as above and suppose that $Z_{M}$ is not empty. Then  
the family of bounded linear maps in $L^{2}(\RE^{n})$ 
\be\label{resolvent}
R_{z}^{\Lambda}:=R^0_{z}+G_{z}P^{*}(PM_{z}P^{*})^{-1}PG^{*}_{\bar z}\,,\qquad z\in Z_{M}\,,
\ee
is the resolvent of a singular perturbation $\Delta_{\Lambda}$ of $\Delta$ and $Z_{M}=\rho(\Delta_{\Lambda})\cap\CO\backslash(-\infty,0]$, $\rho(\Delta_{\Lambda})$ denoting the resolvent set of $\Delta_{\Lambda}$. 
\end{theorem}
In the next sections, given an open, bounded set $\Omega\equiv\Omega_{\-}\subset\RE^{n}$ with a Lipschitz boundary $\Gamma$ and such that  $\Omega_{\+}:=\RE^{n}\backslash\overline\Omega$ is connected, we consider models where the map $\tau:H^{2}(\RE^{n})\to\K$ corresponds to the two different cases:
$$\text{1) \centerline{$\tau= \gamma_{0}\,,\qquad\K=B^{3/2}_{2,2}(\Gamma)\,,\qquad\X=H^{s}(\Gamma)$, $|s|\le 1$;}}
$$
$$\text{2)\centerline{$\tau= \gamma_{1}\,,\qquad \K=H^{1/2}(\Gamma)\,,\qquad\X=H^{s}(\Gamma)$, $-1\le s<1/2$;}}
$$
Here $\K_{0}=L^{2}(\Gamma)$, $\gamma_{0}$ and $\gamma_{1}$ denote the Dirichlet and Neumann traces on the boundary $\Gamma$  and $H^{s}(\Gamma)$, $|s|\le 1$, denotes the Hilbert space of Sobolev functions of order $s$ on $\Gamma$ (see, e.g., \cite[Chapter 3]{McL}); the Hilbert space $B^{3/2}_{2,2}(\Gamma)$ is a Besov-like space (see \cite[Section 2, Chapter V]{JW} for the precise definitions) giving the correct trace space of $\gamma_{0}|H^{2}(\RE^{n})$ in the case $\Gamma$ is a Lipschitz manifold (whenever $\Gamma$ is more regular  it identifies with $H^{3/2}(\Gamma)$).\par 
We introduce the label $\sharp=D, N$  according to the two possible different choices above. The operators defined in \eqref{Gz}  in terms of one of the two traces $\tau_{D}=\gamma_{0}$, $\tau_{N}=\gamma_{1}$  is then denoted by 
$$
G^{\sharp}_{z}:\K^{*}\to L^{2}(\RE^{n})\,,\quad z\in\CO\backslash(-\infty,0]\,.
$$ 
We also introduce the two spaces $\K=\K_{\sharp}$, $\X=\X^{s}_{\sharp}$, where
$$
\K_{D}=B^{3/2}_{2,2}(\Gamma)\,,\qquad\K_{N}=H^{1/2}(\Gamma)\,,
$$
$$
\X_{D}^{s}:=H^{1/2-s}(\Gamma)\,,\qquad \X_{N}^{s}:=H^{-1/2-s}(\Gamma) \,,\quad 0\le s\le 1/2\,.
$$
Since $\K_{\sharp}\hookrightarrow \X^{s}_{\sharp}$, and hence ${\X^{s}_{\sharp}}^{*}\hookrightarrow \K_{\sharp}^{*}$, we can restrict $G^{\sharp}_{\lambda}$ to the spaces 
${\X^{s}_{\sharp}}^{*}$; notice that, by the definition \eqref{Gz}, one has
$$
G^{D}_{z}\phi=\SL_{z}\phi\,,\qquad G^{N}_{z}\varphi=\DL_{z}\varphi\,,
$$
where $\SL_{z}$ and $\DL_{z}$ are the single- and double-layer operators associated to $(-\Delta+z)$ (see, e.g., \cite[Chapter 6]{McL}).
\subsection{Laplace operators with boundary conditions on hypersurfaces.}\label{closed} In this section we use Theorem \ref{WO} in the case $\Y=\X$ and $P=1_{\X}$.
\subsubsection{The Dirichlel Laplacian.}\label{Dir} Let $\Delta^{D}_{\Omega_{\-/\+}}$ be the self-adjoint operators in $L^{2}(\Omega_{\-/\+})$ corresponding to the Laplace operator with Dirichlet boundary conditions. One has $\Delta^{D}_{\Omega_{\-}}\oplus \Delta^{D}_{\Omega_{\+}}= \Delta_{\Lambda^{D}}$, where
$$
\Lambda^{D}_{z}=(M_{z}^{D})^{-1}=-(\gamma_{0}\SL_{z})^{-1}\in\B(H^{1/2}(\Gamma),H^{-1/2}(\Gamma))\,,\quad z\in\CO\backslash(-\infty,0]
$$
(see \cite[Section 5.2]{JMPA}).
\subsubsection{The Neumann Laplacian.}\label{Neu}  Let $\Delta^{N}_{\Omega_{\-/\+}}$ be the self-adjoint operators in $L^{2}(\Omega_{\-/\+})$ corresponding to the Laplace operator with Neumann boundary conditions. One has  $\Delta^{N}_{\Omega_{\-}}\oplus \Delta^{N}_{\Omega_{\+}}= \Delta_{\Lambda^{N}}$, where 
$$\Lambda^{N}_{z}=(M_{z}^{N})^{-1}=-(\gamma_{1}\DL_{z})^{-1}\in\B(H^{-1/2}(\Gamma),H^{1/2}(\Gamma))\,,\quad z\in\CO\backslash(-\infty,0]
$$ 
(see \cite[Section 5.3]{JMPA}).
\subsubsection{The Laplacian with semitransparent boundary conditions.}\label{ST} Here $\alpha$ and $\theta$ are real-valued functions and we use the same symbols to denote the corresponding multiplication operators. Taking
$$
\Lambda^{\alpha}_{z}:=(M^{\alpha}_{z})^{-1}=-\left(\frac1\alpha+\gamma_{0}\SL_{z}\right)^{-1}\in\B(L^{2}(\Gamma))\,,\quad z\in\CO\backslash(-\infty,\lambda_{\alpha}]\,,\ \lambda_{\alpha}\ge 0\,,
$$
where $\alpha\in L^{\infty}(\Gamma)$, $\frac1\alpha\in L^{\infty}(\Gamma)$, the self-adjoint operator $\Delta_{\Lambda^{\alpha}}$ represents a bounded from above Laplace operator with the semi-transparent boundary conditions 
$$
\alpha\gamma_{0}u=[\gamma_{1}]u\,,\quad [\gamma_{0}]u=0
$$
(see \cite[Section 4.1.3]{MP-inv}); here $[\gamma_{0}]u$ and $[\gamma_{1}]u$ denote the jump of the Dirichlet and Neumann traces of the function $u$ across the boundary. Taking 
\be\label{delta}
\Lambda^{\theta}_{z}:=(M^{\theta}_{z})^{-1}=(\theta-\gamma_{1}\DL_{z})^{-1}\in\B(H^{-1/2}(\Gamma),H^{1/2}(\Gamma))\,,\quad z\in\CO\backslash(-\infty,\lambda_{\theta}]\,,\ \lambda_{\theta}\ge 0\,,
\ee
where $\theta\in L^{p}(\Gamma)$, $p>2$, the self-adjoint operator $\Delta_{\Lambda^{\theta}}$ represents a bounded from above Laplace operator with the semi-transparent boundary conditions 
\be\label{delta'}
\gamma_{1}u=\theta[\gamma_{0}]u\,,\quad [\gamma_{1}]u=0
\ee
(see \cite[Section 4.1.4]{MP-inv}).

\subsection{Laplace operators with boundary conditions on unclosed hypersurfaces.}\label{unclosed}  Here we consider the case where the boundary conditions are assigned on a relatively open subset $\Sigma$ of the boundary $\Gamma$ of the domain $\Omega$. We use Theorem \ref{WO} with $$
\Y=H^{s}(\Sigma)\,,\qquad P=R_{\Sigma}:H^{s}(\Gamma)\to H^{s}(\Sigma)\,,\quad R_{\Sigma}\phi:=\phi|\Sigma\,.
$$  
Notice that $\Y^{*}=H^{s}(\Sigma)^{*}$ identifies with $H^{-s}_{\overline\Sigma}(\Gamma)$ and $P^{*}$ identifies with 
$$
R^{*}_{\Sigma}:H^{-s}_{\overline\Sigma}(\Gamma)\to H^{-s}(\Gamma)\,,\qquad R^{*}_{\Sigma}\phi:=\phi
$$ 
(see \cite[Lemma 5.1]{MP-inv}), where 
$$
H^{s}_{\overline\Sigma}(\Gamma):=\{\phi\in H^{s}(\Gamma):\supp(\phi)\subseteq \overline\Sigma\}\,.
$$
\subsubsection{The Dirichlel Laplacian.}\label{Dir-s} Considering $\Lambda^{D,\Sigma}_{z}:=R^{*}_{\Sigma}(R_{\Sigma}M^{D}_{z}R^{*}_{\Sigma})^{-1}R_{\Sigma}$, $z\in\CO\backslash(-\infty,0]$, 
$$
(R_{\Sigma}M^{D}_{z}R^{*}_{\Sigma})^{-1}
=-(R_{\Sigma}\gamma_{0}\SL_{z}R^{*}_{\Sigma})^{-1}\in\B(H^{1/2}(\Sigma),H_{\overline \Sigma}^{-1/2}(\Gamma))\,,
$$
the self-adjoint operator $\Delta_{\Lambda^{D,\Sigma}}$ represents a bounded from below Laplacian with Dirichlet boundary conditions on $\Sigma$ (see \cite[Example 7.1]{JST}, \cite[Section 5.1.1]{MP-inv}).
\subsubsection{The Neumann Laplacian.}\label{Neu-s} Considering $\Lambda^{N,\Sigma}_{z}:=R^{*}_{\Sigma}(R_{\Sigma}M^{N}_{z}R^{*}_{\Sigma})^{-1}R_{\Sigma}$, $z\in\CO\backslash(-\infty,0]$, 
$$
(R_{\Sigma}M_{z}^{N}R_{\Sigma}^{*})^{-1}=-(R_{\Sigma}\gamma_{1}\DL_{z}R_{\Sigma}^{*})^{-1}\in\B(H^{-1/2}(\Sigma),H_{\overline\Sigma}^{1/2}(\Gamma))\,,
$$
the self-adjoint operator $\Delta_{\Lambda^{N,\Sigma}}$ represents a bounded from below Laplacian with Neumann boundary conditions on $\Sigma$ (see \cite[Example 7.1]{JST}, \cite[Section 5.1.2]{MP-inv}).
\subsubsection{The Laplacian with semitransparent boundary conditions.}\label{ST-s} Considering $\Lambda^{\alpha,\Sigma}_{z}:=\newline R^{*}_{\Sigma}(R_{\Sigma}M^{\alpha}_{z}R^{*}_{\Sigma})^{-1}R_{\Sigma}$, $z\in\CO\backslash(-\infty,\tilde\lambda_{\alpha}]$, $\tilde\lambda_{\alpha}\ge 0$, 
$$
(R_{\Sigma}M^{\alpha}_{z}R_{\Sigma}^{*})^{-1}=-\left(R_{\Sigma}\left(\frac1\alpha+\gamma_{0}\SL_{z}\right)R_{\Sigma}^{*}\right)^{-1}\in\B(L^{2}(\Sigma);L_{\overline\Sigma}^{2}(\Gamma))\,,
$$
where $\mbox{sign}(\alpha)$ is constant,   
the self-adjoint operator $\Delta_{\Lambda^{\alpha,\Sigma}}$ represents a bounded from below Laplacian with boundary conditions \eqref{delta} on $\Sigma$ (see \cite[Example 7.3]{JST}, \cite[Section 5.1.3]{MP-inv}).\par
Considering $\Lambda^{\theta,\Sigma}_{z}:=R^{*}_{\Sigma}(R_{\Sigma}M^{\theta}_{z}R^{*}_{\Sigma})^{-1}R_{\Sigma}$, $z\in\CO\backslash(-\infty,\tilde\lambda_{\theta}]$, $\tilde\lambda_{\theta}\ge 0$,  
$$
(R_{\Sigma}M^{\theta}_{z}R_{\Sigma}^{*})^{-1}=(R_{\Sigma}(\theta-\gamma_{1}\DL_{z})R_{\Sigma}^{*})^{-1}\in\B(H^{-1/2}(\Sigma),H_{\overline\Sigma}^{1/2}(\Gamma))\,,
$$
the self-adjoint operator $\Delta_{\Lambda^{\theta,\Sigma}}$ represents a bounded from below Laplacian with boundary conditions \eqref{delta'} on $\Sigma$ (see \cite[Example 7.4]{JST}, \cite[Section 5.1.4]{MP-inv}).
\section{Wave Scattering in the Laplace domain.}
Let  $\Delta_{\Lambda}\le \lambda_{\Lambda}$, $\lambda_{\Lambda}\ge 0$, be a semi-bounded singular perturbation in $L^{2}(\RE^{n})$ as defined in the previous section; we consider the Cauchy problem for the wave equation
\be
\begin{cases}\label{Cpb}
\partial_{tt} u(t)=\Delta_{\Diamond}u(t)\\
u(0)=u_{0}\\
\partial_{t} u(0)=v_{0}\,.
\end{cases}
\ee
Here the index $\Diamond$ has the two possible values $\Diamond=0$ or $\Diamond=\Lambda$ and $\Delta_{0}$ identifies with the free Laplacian $\Delta:H^{2}(\RE^{n})\subset L^{2}(\RE^{n})\to L^{2}(\RE^{n})$; in the following we set $\lambda_{0}=0$.\par We say that $u\in C(\RE_{+};L^{2}(\RE^{n}))$ is a mild solution of \eqref{Cpb} whenever
$$
\int_{0}^{t}\int_{0}^{s}u(r)\,dr\,ds\equiv\int_{0}^{t}(t-s)u(s)\,ds\,\in\dom(\Delta_{\Diamond})
$$
and
$$
u(t)=u_{0}+tv_{0}+\Delta_{\Diamond}\int_{0}^{t}(t-s)u(s)\,ds
$$
for any $t\ge 0$. 
By \cite[Proposition 3.14.4, Corollary 3.14.8 and Example 3.14.16]{A}, the unique mild solution of \eqref{Cpb} is given by 
\be\label{sol}
u(t)=\mbox{Cos}_{\Diamond}(t)\,u_{0}+\mbox{Sin}_{\Diamond}(t)\,v_{0}
\ee
where the $\bou(L^{2}(\RE^{n}))$-valued  functions $t\mapsto\mbox{Cos}_{\Diamond}(t)$ and $t\mapsto\mbox{Sin}_{\Diamond}(t)$ are univocally defined through the $\bou(L^{2}(\RE^{n}))$-valued (inverse) Laplace transform by the relations  
\be\label{cos}
\sqrt\lambda\,(-\Delta_{\Diamond}+\lambda)^{-1}=\int_{0}^{\infty}e^{-\sqrt\lambda\, t}\,\mbox{Cos}_{\Diamond}(t)\,dt\,,\qquad \lambda>\lambda_{\Diamond}\,,
\ee
\be\label{sin}
(-\Delta_{\Diamond}+\lambda)^{-1}=\int_{0}^{\infty}e^{-\sqrt\lambda\, t}\,\mbox{Sin}_{\Diamond}(t)\,dt\,,
\qquad\lambda>\lambda_{\Diamond}\,.
\ee
One has (see \cite[(6.14), (6.15), Chap. II]{Fatto})
\be\label{hyp}
\|\Cos_{\Diamond}(t)\|_{\bou(L^{2}(\RE^{n}))}\le\cosh(\sqrt{\lambda_{\Diamond}}\,t)\,,\quad 
\|\Sin_{\Diamond}(t)\|_{\bou(L^{2}(\RE^{n}))}\le\frac{\sinh(\sqrt{\lambda_{\Diamond}}\,t)}{\sqrt{\lambda_{\Diamond}}}\,.
\ee
Whenever $\lambda_{\Diamond}=0$, by functional calculus one gets
$$
\mbox{Cos}_{\Diamond}(t)=\cos(t(-\Delta_{\Diamond})^{1/2})\,,\qquad \mbox{Sin}_{\Diamond}(t)=
(-\Delta_{\Diamond})^{-1/2}\sin(t(-\Delta_{\Diamond})^{1/2})
$$ 
and so, in this case,
\be\label{trig}
\|\Cos_{\Diamond}(t)\|_{\bou(L^{2}(\RE^{n}))}\le 1\,,\quad 
\|\Sin_{\Diamond}(t)\|_{\bou(L^{2}(\RE^{n}))}\le \min\{t,1\}\,.
\ee
Given $\chi_{\varepsilon}$ a bounded, not negative function such that 
$$\text{$\chi_{\varepsilon}(s)=0$ whenever $s\ge \varepsilon>0$ and $\int_{0}^{\varepsilon}\chi_{\varepsilon}(s)\,ds=1$,}
$$ and given $f\in L^{2}(\RE^{n})$,  let $u^{\Diamond}_{f,\varepsilon}$ be the solution of the wave equation with the pulse $\chi_{\varepsilon}f$, i.e., $u^{\Diamond}_{f,\varepsilon}$ solves the inhomogeneous Cauchy problem 
\be
\begin{cases}\label{Cpb-inh}
\partial_{tt} u^{\Diamond}_{f,\varepsilon}(t)=\Delta_{\Diamond}u^{\Diamond}_{f,\varepsilon}(t)+\chi_{\varepsilon}(t)f\\
u^{\Diamond}_{f,\varepsilon}(0)=0\\
\partial_{t} u^{\Diamond}_{f,\varepsilon}(0)=0\,.
\end{cases}
\ee
By \cite[Proposition 3.1.16]{A} (see also \cite[Section II.4]{Fatto}), $u^{\Diamond}_{f,\varepsilon}$ is given by
$$
u^{\Diamond}_{f,\varepsilon}(t):=\int_{0}^{t}\mbox{Sin}_{\Diamond}(t-s)\chi_{\varepsilon}(s)f\,ds
\,.
$$
In scattering experiments one measures the scattered wave 
$$S^{\Lambda,\varepsilon}_{t}f:=u^{\Lambda}_{f,\varepsilon}(t)-u^{0}_{f,\varepsilon}(t)$$
produced by the short pulse $\chi_{\varepsilon}f$, $\varepsilon\ll 1$. Since the measurements last a finite time $t_{\circ}\gg \varepsilon$ and  
detectors occupy a finite region, we introduce  the continuous $\bou(L^{2}(B))$-valued map
\be\label{fp}
t\mapsto 1_{[0,t_{\circ}]}(t)1_{B}S^{\Lambda,\varepsilon}_{t}1_{B}\,,
\ee
where $B\subset\subset\RE^{n}\backslash\overline\Omega$ is open and bounded and $1_{X}$ denotes the indicator function of a set $X$. We introduce 
the operator family 
$F^{\Lambda,t_{\circ},\varepsilon}_{\lambda}$ given by 
\be\label{FFF}
F^{\Lambda,t_{\circ},\varepsilon}_{\lambda}:=\int_{0}^{t_{\circ}}e^{-\sqrt{\lambda}\, t}\, 1_{B}S^{\Lambda,\varepsilon}_{t}1_{B}\,dt\,,\quad \lambda>\lambda_{\Lambda}\,,
\ee
which is the Laplace transform of \eqref{fp}. In an ideal setup, corresponding to instantaneous pulses and measurement lasting an infinite time, \eqref{fp} rephrases as 
$$
t\mapsto 1_{B}S^{\Lambda}_{t}1_{B}\,,\qquad S^{\Lambda}_{t}:=\Sin_{\Lambda}(t)-\Sin_{0}(t)\,.
$$
By Laplace transform again, we define the ideal operator
\be\label{FF}
F^{\Lambda}_{\lambda}:=\int_{0}^{\infty}e^{-\sqrt{\lambda}\, t} 1_{B}S^{\Lambda}_{t}1_{B}\,dt\,,\quad \lambda>\lambda_{\Lambda}\,.
\ee
The next Lemma shows that for any given $\varepsilon>0$ and $t_{\circ}>0$, one can choose a sufficiently large $\lambda$ such that the difference $F^{\Lambda}_{\lambda}-F^{\Lambda,t_{\circ},\varepsilon}_{\lambda}$ is as small (in uniform operator norm) as one likes. Taking into account \eqref{hyp} and \eqref{trig}, here we set $x^{-1}\sinh xt=\min\{t,1\}$ whenever $x=0$.
\begin{lemma}\label{estimate} Let  $\Delta_{\Lambda}\le \lambda_{\Lambda}$, $\lambda_{\Lambda}\ge 0$, be a semi-bounded singular perturbation in $L^{2}(\RE^{n})$ as defined in Theorem \ref{WO}.
For any $\lambda$ and $\lambda^{\circ}_{\Lambda}$ such that $\lambda\ge\lambda^{\circ}_{\Lambda}>\lambda_{\Lambda}$ one has 
\begin{align*}
\|F^{\Lambda}_{\lambda}-F^{\Lambda,t_{\circ},\varepsilon}_{\lambda}\|_{\bou(L^{2}(\RE^{n}))}\le \,\frac1{\sqrt\lambda}\left(c_{1}\,e^{-\sqrt\lambda\,t_{\circ}}+\varepsilon\left(c_{2}(1-e^{-\sqrt\lambda\,\varepsilon})+c_{3}\,e^{-\sqrt\lambda\,\varepsilon}\right)\right)
\end{align*}
where
$$
c_{1}=\frac{\lambda_\Lambda^{\circ}}{\lambda^{\circ}_{\Lambda}-\lambda_{\Lambda}}\left(\frac{\cosh(\sqrt{\lambda_{\Lambda}}\,t_{\circ})}{\sqrt{\lambda_\Lambda^{\circ}}}+ \frac{\sinh(\sqrt{\lambda_{\Lambda}}\,t_{\circ})}{\sqrt{\lambda_{\Lambda}}}\,\right)
+\left(\frac1{\sqrt{\lambda_\Lambda^{\circ}}}+\min\{t_{\circ},1\}\right)\,,
$$
$$
c_{2}=\cosh (\sqrt{\lambda_{\Lambda}}\,\varepsilon)+\frac{\sinh(\sqrt{\lambda_{\Lambda}}\,\varepsilon)}{\sqrt{\lambda_{\Lambda}}\,\varepsilon}+2\,,
$$
$$
c_{3}=\cosh (\sqrt{\lambda_{\Lambda}}\,t_{\circ})+1\,.
$$
\end{lemma}
\begin{proof} Let us re-write the difference $F^{\Lambda}_{\lambda}-F^{\Lambda,t_{\circ},\varepsilon}_{\lambda}$
as
\begin{align*}
&F^{\Lambda}_{\lambda}-F^{\Lambda,t_{\circ},\varepsilon}_{\lambda}=\int_{t_{\circ}}^{\infty}e^{-\sqrt\lambda\, t}\,1_{B}S^{\Lambda}_{t}1_{B}\,dt\\
+&\int_{0}^{\varepsilon}e^{-\sqrt\lambda\, t}\,1_{B}(S^{\Lambda}_{t}-S^{\Lambda,\varepsilon}_{t})1_{B}\,dt+\int_{\varepsilon}^{t_{\circ}}e^{-\sqrt\lambda\, t}\,1_{B}(S^{\Lambda}_{t}-S^{\Lambda,\varepsilon}_{t})1_{B}\,dt\\
=&I_{1}+I_{2}+I_{3}\,.
\end{align*}
By \eqref{cos}, \eqref{sin} and by the identity (see \cite[equation (3.95)]{A})
$$
\Sin_{\Diamond}(t+{t_{\circ}})=\Cos_{\Diamond}({t_{\circ}})\Sin_{\Diamond}(t)+\Sin_{\Diamond}({t_{\circ}})\Cos_{\Diamond}(t)\,,
$$
one gets
\begin{align*}
I_{1}=&e^{-\sqrt\lambda\, {t_{\circ}}}
\int_{0}^{\infty}e^{-\sqrt\lambda\, t}\,1_{B}S^{\Lambda}_{t+{t_{\circ}}}1_{B}\,dt\\
=&e^{-\sqrt\lambda\,{t_{\circ}}}\,1_{B}\big((\Cos_{\Lambda}({t_{\circ}})+\sqrt\lambda\,\Sin_{\Lambda}({t_{\circ}}))R^{\Lambda}_{\lambda}-(\Cos_{0}({t_{\circ}})+\sqrt\lambda\,\Sin_{0}({t_{\circ}}))R^{0}_{\lambda}\big)1_{B}\,.
\end{align*}
By $\|R^{\Diamond}_{\lambda}\|_{\bou(L^{2}(\RE^{n}))}\le (\lambda-\lambda_{\Diamond})^{-1}$ and by \eqref{hyp}, \eqref{trig}, one gets
\begin{align*}
&\|I_{1}\|_{\bou(L^{2}(\RE^{n}))}\\
\le &e^{-\sqrt\lambda\, {t_{\circ}}}\left(\frac1{\lambda-\lambda_{\Lambda}}\left(\cosh(\sqrt{\lambda_{\Lambda}}\,t_{\circ})+\sqrt{\lambda}\ \frac{\sinh(\sqrt{\lambda_{\Lambda}}\,t_{\circ})}{\sqrt{\lambda_{\Lambda}}}\right)
+\frac1{\lambda}\left(1+\sqrt{\lambda}\,\min\{t_{\circ},1\}\right)
\right)\\
\le &\frac{e^{-\sqrt\lambda\, {t_{\circ}}}}{\sqrt\lambda}\left(\frac1{1-\lambda_{\Lambda}/\lambda}\left(\frac{\cosh(\sqrt{\lambda_{\Lambda}}\,t_{\circ})}{\sqrt\lambda}+ \frac{\sinh(\sqrt{\lambda_{\Lambda}}\,t_{\circ})}{\sqrt{\lambda_{\Lambda}}}\right)
+\left(\frac1{\sqrt\lambda}+\min\{t_{\circ},1\}\right)
\right)
\end{align*}
By 
$$\mbox{Sin}_{\Diamond}(t)=\int_{0}^{t}\mbox{Cos}_{\Diamond}(s)\,ds
$$
(see \cite[equation (3.93)]{A}), 
one gets
\begin{align*}
&\|u^{\Diamond}_{f,\varepsilon}(t)-\mbox{Sin}_{\Diamond}(t)f\|_{L^{2}(\RE^{n})}\\
\le& \int_{0}^{t}\|\mbox{Sin}_{\Diamond}(t-s)-\mbox{Sin}_{\Diamond}(t)\|_{\bou(L^{2}(\RE^{n}))}\,\chi_{\varepsilon}(s)\,ds\ \|f\|_{L^{2}(\RE^{n})}\\
&+\|\mbox{Sin}_{\Diamond}(t)\|_{\bou(L^{2}(\RE^{n}))}\left(1-\int_{0}^{t}\chi_{\varepsilon}(s)\,ds\right)\|f\|_{L^{2}(\RE^{n})}\,
\\
\le&
\left(\varepsilon \,\cosh (\sqrt{\lambda_{\Diamond}}\,t)\int_{0}^{t}\chi_{\varepsilon}(s)\,ds+\frac{\sinh(\sqrt{\lambda_{\Diamond}}\,t)}{\sqrt{\lambda_{\Diamond}}}\left(1-\int_{0}^{t}\chi_{\varepsilon}(s)\,ds\right)\right)\|f\|_{L^{2}(\RE^{n})}\,.
\end{align*}
Thus
$$
\|I_{2}\|_{\bou(L^{2}(\RE^{n}))}\le \left(\varepsilon \,\cosh (\sqrt{\lambda_{\Lambda}}\,\varepsilon)+\frac{\sinh(\sqrt{\lambda_{\Lambda}}\,\varepsilon)}{\sqrt{\lambda_{\Lambda}}}+2\varepsilon\right)\,\frac{1-e^{-\sqrt\lambda\,\varepsilon}}{\sqrt\lambda}
$$
and
$$
\|I_{3}\|_{\bou(L^{2}(\RE^{n}))}\le \left(\varepsilon \,\cosh (\sqrt{\lambda_{\Lambda}}\,t_{\circ})+\varepsilon\right)\,\frac{e^{-\sqrt\lambda\,\varepsilon}-e^{-\sqrt\lambda\,t_{\circ}}}{\sqrt\lambda}\,.
$$

\end{proof}

\section{Inverse scattering in the time domain.}

\begin{lemma}\label{compact} Let $\Lambda_{z}=P^{*}(PM_{z}P^{*})^{-1}P$ define the self-adjoint operator $\Delta_{\Lambda}$ as in Theorem \ref{WO} and assume that the embedding $\ran(\Lambda_{\lambda})\hookrightarrow\K^{*}$ be compact. Let $\lambda>\lambda_{\Lambda}$, $B\subset\subset\RE^{n}\backslash\overline\Omega$ open and bounded,  and let $F^{\Lambda}_{\lambda}$ be defined as in \eqref{FF}. Then
\be\label{spectrum}
\sigma_{\rm disc}(  F_{\lambda}^{\Lambda})=\sigma(  F_{\lambda}^{\Lambda})  \backslash\{  0\}
=\{\mu^{\Lambda}_{\lambda,k}\}_{1}^{\infty}\subset\RE\backslash\{0\}\,,\quad\lim_{k\to\infty}\mu^{\Lambda}_{\lambda,k}=0\,,
\ee
and there exists an orthonormal sequence $\{v^{\Lambda}_{\lambda,k}\}_{1}^{\infty}\subset L^{2}(B)$ such that, for every $u\in L^{2}(B)$, 
\be\label{sp-res-1}
u=u_{0}+\sum_{k=1}^{\infty}\langle v_{\lambda,k}^{\Lambda},u\rangle_{L^{2}(B)}\,v_{\lambda,k}^{\Lambda}\,,\quad\text{ $u_{0}\in\ker(F_{\lambda}^{\Lambda})$.}
\ee
Moreover
\be\label{sp-res-2}
F_{\lambda}^{\Lambda}
=\sum_{k=1}^{\infty}\mu_{\lambda,k}^{\Lambda}\,v_{\lambda,k}^{\Lambda}\otimes v_{\lambda,k}^{\Lambda}\,.
\ee
\end{lemma}
\begin{proof} By \eqref{sin} and the resolvent formula \eqref{resolvent}, one gets $F^{\Lambda}_{\lambda}=1_{B}G_{\lambda}\Lambda_{\lambda}G^{*}_{\lambda}1_{B}$. Thus, the compactness of  the embedding $\ran(\Lambda_{\lambda})\hookrightarrow\K^{*}$ implies that $F^{\Lambda}_{\lambda}$ is compact. Hence \eqref{spectrum}, \eqref{sp-res-1}, \eqref{sp-res-2} are consequence of the spectral theory for compact self-adjoint operators (see, e.g., \cite[Section 6]{Jorg}).
\end{proof}
\subsection{Obstacles reconstruction.}
Before stating our results, let us introduce the following family of functions in 
$L^{2}(\RE^{n})$:
\be\label{test}
g_{\lambda}^{x}(y):=\frac{\lambda^{n/2-1}}{(2\pi)^{n/2}}\,\frac{K_{\frac{n-2}2}(\sqrt\lambda\,\|x-y\|)}{\|\sqrt\lambda\,(x-y)\|^{n/2-2}}\,,
\ee
where $K_{\nu}$ denotes the modified Bessel function of the third kind of order $\nu$. Notice that  $g^{x}_{\lambda}$ identifies with the fundamental solution of $(-\Delta+\lambda)$; in particular, whenever $n=3$, 
$$
g_{\lambda}^{x}(y):=\frac{e^{-\sqrt\lambda\,\|x-y\|}}{4\pi\|x-y\|} \,.
$$
\begin{theorem}\label{thm} Let $\Delta_{\Lambda}\le\lambda_{\Lambda}$, $\lambda_{\Lambda}\ge 0$, be defined as in Theorem \ref{WO} with $M_{\lambda}\in\bou({\X^{s}_{\sharp}}^{*},\X^{s}_{\sharp})$, 
$\lambda>\lambda_{\Lambda}$, and 
$P=1_{\X^{s}_{\sharp}}$.
Given $\lambda>\lambda_{\Lambda}$, $B\subset\subset\RE^{n}\backslash\overline\Omega$ open and bounded,  let $F^{\Lambda}_{\lambda}$ be defined as in \eqref{FF}. \par
If $M_{\lambda}$ is coercive, i.e., there exists $c_{\lambda}>0$ such that 
\be\label{coerc}
\forall\phi\in{\X^{s}_{\sharp}}^{*}\,,\qquad\big|\langle \phi,M_{\lambda}\phi\rangle_{ {\X^{s}_{\sharp}}^{*}\!,\X^{s}_{\sharp}}\big|\ge c_{\lambda}\,\|\phi\|^{2}_{{\X^{s}_{\sharp}}^{*}} \,,
\ee
then
\be\label{inf}
x\in\Omega\iff\inf_{\substack{u\in L^{2}(B)\\ \langle u,g^{x}_{\lambda}\rangle_{L^{2}(B)}=1}}\left|\langle u,F^{\Lambda}_{\lambda}u\rangle_{L^{2}(B)}\right|>0\,;
\ee
if $M_{\lambda}$ is sign-definite, i.e., there exists $c_{\lambda}>0$ such that one of the two inequalities
\be\label{s-d}
\forall\phi\in{\X^{s}_{\sharp}}^{*}\,,\qquad \pm\langle \phi,M_{\lambda}\phi\rangle_{ {\X^{s}_{\sharp}}^{*}\!,\X^{s}_{\sharp}}\ge c_{\lambda}\,\|\phi\|^{2}_{{\X^{s}_{\sharp}}^{*}} 
\ee
 holds, then
\be\label{1/4}
x\in\Omega\iff\sum_{k=1}^{\infty}\,\frac{|\langle g^{x}_{\lambda},v^{\Lambda}_{\lambda,k}\rangle_{L^{2}(B)}|^{2}}{|\mu^{\Lambda}_{\lambda,k}|}<+\infty\,,
\ee
where the sequences $\{\mu^{\Lambda}_{\lambda,k}\}_{1}^{\infty}$, $\{v^{\Lambda}_{\lambda,k}\}_{1}^{\infty}$ providing the spectral resolution of $F^{\Lambda}_{\lambda}$ are given  in Lemma \ref{compact}.
\end {theorem}
\begin{proof}
By \eqref{sin} and by the resolvent formula \eqref{resolvent}, one gets the factorized representation
\be\label{factor}
F^{\Lambda}_{\lambda}=1_{B}G^{\sharp}_{\lambda}M^{-1}_{\lambda}{G_{\lambda}^{\sharp}}^{*}1_{B}=(1_{B}G^{\sharp}_{\lambda}M^{-1}_{\lambda})M_{\lambda}(1_{B}G_{\lambda}^{\sharp}M_{\lambda}^{-1})^{*}\,.
\ee
Then \eqref{inf} is consequence, by the inf-criterion in \cite[Theorem 1.16]{KG}, of the coercivity hypothesis about $M_{\lambda}$ and Lemma \ref{FM} in the Appendix.\par
Suppose now that $M_{\lambda}>0$ (the case $M_{\lambda}<0$ is similar, simply replace $F^{\Lambda}_{\lambda}$ by $-F^{\Lambda}_{\lambda}$) then, by \eqref{factor} and \cite[Corollary 1.22]{KG}, one has that 
$\ran(1_{B}G_{\lambda}M^{-1}_{\lambda})=\ran(1_{B}G_{\lambda})=\ran((F^{\Lambda}_{\lambda})^{1/2})$.  Thus, by Lemma \ref{FM}, $x\in\Omega$ if and only if $1_{B}g^{x}_{\lambda}\in\ran((F^{\Lambda}_{\lambda})^{1/2})=\dom((F^{\Lambda}_{\lambda})^{-1/2})$. By \eqref{sp-res-2}, the latter is equivalent to the convergence of the series in \eqref{1/4}. Indeed Lemma \ref{compact} applies since $\ran(\Lambda_{\lambda})=\dom(M_{\lambda})=
{X_{\sharp}^{s}}^{*}$ is compactly embedded in $\K^{*}_{\sharp}$.
\end{proof}
\subsection{Screens reconstruction.} Here we consider the case where the boundary conditions are assigned on a relatively open subset $\Sigma$ of the boundary $\Gamma$ of the domain $\Omega$. 
We introduce the spaces
$$
\widetilde\X_{D}^{s}:=H^{1/2-s}(\Sigma)\,,\qquad \widetilde\X_{N}^{s}:=H^{-1/2-s}(\Sigma) \,,$$
so that 
$$
(\widetilde\X_{D}^{s})^{*}:=H^{s-1/2}_{\overline\Sigma}(\Gamma)\,,\qquad (\widetilde\X_{N}^{s})^{*}:=H^{s+1/2}_{\overline\Sigma}(\Gamma) \,,\quad 0\le s\le 1/2\,.
$$
\begin{theorem}\label{thm-s} Let $\Delta_{\Lambda}\le\lambda_{\Lambda}$, $\lambda_{\Lambda}\ge 0$, be defined as in Theorem \ref{WO} with $M_{\lambda}\in\bou({\X^{s}_{\sharp}}^{*},\X^{s}_{\sharp})$, 
$\lambda>\lambda_{\Lambda}$, $\Y=\widetilde\X^{s}_{\sharp}$ and
$P=R_{\Sigma}$. Given $\lambda>\lambda_{\Lambda}$, $B\subset\subset\RE^{n}\backslash\overline\Omega$ open and bounded,  let $F^{\Lambda}_{\lambda}$ be defined as in \eqref{FF}. Let $\Sigma_{\circ}\subset\Gamma_{\!\circ}$ be a relatively open subset, with a Lipschitz boundary, of the Lipschitz hypersurface $\Gamma_{\!\circ}$.\par
If $M_{\lambda}$ is coercive, i.e., \eqref{coerc} holds, then
\be\label{inf-s}
\Sigma_{\circ}\subset\Sigma\iff\inf_{\substack{u\in L^{2}(B)\\ \langle u,g^{\Sigma_{\circ}}_{\lambda}\rangle_{L^{2}(B)}=1}}\left|\langle u,F^{\Lambda}_{\lambda}u\rangle_{L^{2}(B)}\right|>0\,,
\ee
where $g_{\lambda}^{\Sigma_{\circ}}(y):=\int_{\Sigma_{\circ}}g^{x}_{\lambda}(y)\,d\sigma_{\Gamma_{\!\circ}}(x)$; if $M_{\lambda}$ is sign-definite, i.e., one of the two inequalities in \eqref{s-d} holds, then
\be\label{1/4-s}
\Sigma_{\circ}\subset\Sigma\iff\sum_{k=1}^{\infty}\,\frac{|\langle g^{\Sigma_{\circ}}_{\lambda},v^{\Lambda}_{\lambda,k}\rangle_{L^{2}(B)}|^{2}}{|\mu^{\Lambda}_{\lambda,k}|}<+\infty\,,
\ee
where the sequences $\{\mu^{\Lambda}_{\lambda,k}\}_{1}^{\infty}$, $\{v^{\Lambda}_{\lambda,k}\}_{1}^{\infty}$ providing the spectral resolution of $F^{\Lambda}_{\lambda}$ are given in Lemma \ref{compact}.
\end {theorem}
\begin{proof}
By \eqref{sin} and by the resolvent formula \eqref{resolvent}, one gets the factorized representation 
\begin{align*}
&F^{\Lambda}_{\lambda}=1_{B}G^{\sharp}_{\lambda}R^{*}_{\Sigma}(R_{\Sigma}M_{\lambda}R^{*}_{\Sigma})^{-1}R_{\Sigma}{G_{\lambda}^{\sharp}}^{*}1_{B}\\
=&(1_{B}G^{\sharp}_{\lambda}R^{*}_{\Sigma}(R_{\Sigma}M_{\lambda}R^{*}_{\Sigma})^{-1})R^{*}_{\Sigma}M_{\lambda}R_{\Sigma}(1_{B}G_{\lambda}^{\sharp}R^{*}_{\Sigma}(R_{\Sigma}M_{\lambda}R^{*}_{\Sigma})^{-1})^{*}\,.
\end{align*}
Since $R_{\Sigma}:\X_{\sharp}^{s}\to\widetilde\X_{\sharp}^{s}$ is an orthogonal projector (see \cite[Lemma 5.1]{MP-inv}, the coercivity of $M_{\lambda}$ implies the coercivity of  $R^{*}_{\Sigma}M_{\lambda}R_{\Sigma}$; likewise if $M_{\lambda}$ is sign-definite then  $R^{*}_{\Sigma}M_{\lambda}R_{\Sigma}$ is sign-definite as well. Then \eqref{inf-s} is consequence, by the inf-criterion in \cite[Theorem 1.16]{KG}, of Lemma \ref{FM-s} in the Appendix. \par 
By \eqref{factor} and \cite[Corollary 1.22]{KG}, one has that 
$\ran(1_{B}G^{\sharp}_{\lambda}R^{*}_{\Sigma}(R_{\Sigma}M_{\lambda}R^{*}_{\Sigma})^{-1})=\ran(1_{B}G^{\sharp}_{\lambda}R^{*}_{\Sigma})=\ran((F^{\Lambda}_{\lambda})^{1/2})$.  Thus, by Lemma \ref{FM-s}, $\Sigma_{\circ}\subset\Sigma$ if and only if $1_{B}g^{\Sigma_{\circ}}_{\lambda}\in\ran((F^{\Lambda}_{\lambda})^{1/2})=\dom((F^{\Lambda}_{\lambda})^{-1/2})$. By \eqref{sp-res-2}, the latter is equivalent to the convergence of the series in \eqref{1/4-s}. Indeed Lemma \ref{compact} applies since $\ran(\Lambda_{\lambda})=\ran(R^{*}_{\Sigma}(R_{\Sigma}M_{\lambda}R^{*}_{\Sigma})^{-1}R_{\Sigma})=
(\widetilde X_{\sharp}^{s})^{*}$ is compactly embedded in $\K^{*}_{\sharp}$.
\end{proof}

\subsection{Applications.}
\subsubsection{Dirichlet boundary conditions.} By \cite[Lemma 3.2]{JDE}, 
$$
M_{\lambda}^{D}=-\gamma_{0}\SL_{\lambda}\in\B(H^{-1/2}(\Gamma),H^{1/2}(\Gamma))\,,\quad \lambda>0\,,
$$
is negative (and hence coercive); therefore Theorems \ref{thm} and \ref{thm-s} apply to $\Delta_{\Lambda^{D}}$ and $\Delta_{\Lambda^{D,\Sigma}}$ (see Sections \ref{Dir} and \ref{Dir-s}) respectively.  
\subsubsection{Neumann boundary conditions.} By \cite[Lemma 3.2]{JDE}, 
$$
M_{\lambda}^{N}=-\gamma_{1}\DL_{\lambda}\in\B(H^{1/2}(\Gamma),H^{-1/2}(\Gamma))\,,\quad \lambda>0\,,
$$
is positive (and hence coercive); therefore Theorems \ref{thm} and \ref{thm-s} apply to $\Delta_{\Lambda^{N}}$ and $\Delta_{\Lambda^{N,\Sigma}}$ (see Sections \ref{Neu} and \ref{Neu-s}) respectively.  
\subsubsection{Semi-transparent boundary conditions.} Let $\alpha$ and $\theta$ be as in Section \ref{ST}. 
$$M^{\alpha}_{\lambda}=-\left(\frac1\alpha+\gamma_{0}\SL_{\lambda}\right)\in\B(L^{2}(\Gamma))\,,\quad\lambda>\lambda_{\alpha}\,,
$$ 
is negative (and hence coercive) whenever $\alpha$ is positive (since $\gamma_{0}\SL_{\lambda}$ is positive by \cite[Lemma 3.2]{JDE});
\par 
$$M^{\theta}_{\lambda}=(\theta-\gamma_{1}\DL_{\lambda})\in\B(H^{1/2}(\Gamma),H^{-1/2}(\Gamma))\,,\quad \lambda>\lambda_{\theta}\,,
$$ 
is positive (and hence coercive) whenever $\theta$ is positive (since $\gamma_{1}\DL_{\lambda}$ is negative by \cite[Lemma 3.2]{JDE}). 
\par 
Therefore, under these hypotheses,  Theorems \ref{thm} and \ref{thm-s} apply to $\Delta_{\Lambda^{\alpha}}$ and $\Delta_{\Lambda^{\theta}}$, $\Delta_{\Lambda^{\alpha,\Sigma}}$ and $\Delta_{\Lambda^{\theta,\Sigma}}$ (see Sections \ref{ST} and \ref{ST-s}) respectively.   
\section{Appendix.}
\begin{lemma}\label{FM} Let $g^{x}_{\lambda}$, $\lambda>0$, be defined as in \eqref{test} and let $B\subset\subset\RE^{n}\backslash\overline\Omega$. Then
$$
x\in\Omega\quad\iff\quad 1_{B}g^{x}_{\lambda}\in\ran(1_{B}G_{\lambda}^{\sharp}|{\X_{\sharp}^{s}}^{*}
)\,, \quad\sharp=D,N\,.
$$
\end{lemma}
\begin{proof} If $x\in \Omega$ then $g^{x}_{\lambda}$ solves the exterior boundary value problem
in $\Omega_{\+}=\RE^{n}\backslash\overline\Omega$
$$
\begin{cases}
(  -\Delta+\lambda)u=0\\
\gamma^{\+}_{k_{\sharp}}u=\gamma_{k_{\sharp}}g^{x}_{\lambda}\,,
\end{cases}
$$
where $\gamma^{\+}_{k_{\sharp}}$ denotes the one-sided trace and $k_{D}=0$, $k_{N}=1$. Since such a problem has a unique  radiating solution (see, e.g., \cite[ Theorem 9.11]{McL} for the Dirichlet case and \cite[Exercise 9.5]{McL} for the Neumann case) and $(-\Delta+\lambda)G^{\sharp}_{\lambda}\phi=0$ in $\Omega_{\+}$, one gets the equality
$g^{x}_{\lambda}|\Omega_{\+}=(G_{\lambda}^{\sharp}\phi)|\Omega_{\+}$, where $\phi=(\gamma_{k_{\sharp}}G^{\sharp}_{\lambda})^{-1}\gamma_{k_{\sharp}}g^{x}_{\lambda}$.
Here we are using $(\gamma_{0}\SL_{\lambda})^{-1} \in\B(H^{s+1/2}(\Gamma),H^{s-1/2})$ and $(\gamma_{1}\DL_{\lambda})^{-1} \in\B(H^{s-1/2}(\Gamma),H^{s+1/2})$, $s\in[0,1/2]$; this is consequence of bounded invertibility for $s=0$ (see \cite[relations (5.32) and (5.33)]{JMPA} combined with the mapping properties provided in \cite[Theorem 3]{costa}. Suppose now that $x\notin\Omega$ and that there exists $\phi\in\X^{s}_{\sharp}$ such that $1_{B}g^{x}_{\lambda}=1_{B}G^{\sharp}_{\lambda}\phi$. Since both $g^{x}_{\lambda}$ and $G^{\sharp}_{\lambda}\phi$ solve $(-\Delta+\lambda)u=0$ in $\Omega_{\+}\backslash\{x\}$, they coincide there by the principle of unique continuation. This is not possible, since $g^{x}_{\lambda}$ is singular at $x$ while $G^{\sharp}_{\lambda}\phi\in\C^{\infty}(\Omega_{\+})$ by elliptic regularity. 
\end{proof}

\begin{lemma}\label{FM-s} Let $\Sigma_{\circ}\subset\Gamma_{\!\circ}$ be a relatively open subset of $\Gamma_{\!\circ}$, the Lipschitz boundary of an open bounded set $\Omega_{\circ}\subset\RE^{n}$; define $g^{\Sigma_{\circ}}_{\lambda}(y):=\int_{\Sigma_{\circ}}g^{x}_{\lambda}(y)\,d\sigma_{\Gamma_{\!\circ}}(x)$. Then 
$$
\Sigma_{\circ}\subset\Sigma \quad\iff\quad 1_{B}g^{\Sigma_{\circ}}_{\lambda}\in\ran(1_{B}G_{\lambda}^{\sharp}|(\widetilde\X_{\sharp}^{s})^{*})\,, \quad\sharp=D,N\,.
$$
\end{lemma}

\begin{proof} ($\sharp=D$) Let $\Sigma_{\circ}\subset \Sigma$; then $g^{\Sigma_{\circ}}_{\lambda}=\SL_{\lambda}1_{\Sigma_{\circ}}=G_{\lambda}^{D}\phi$ with $\phi=1_{\Sigma_{\circ}}\in H^{1/2-\epsilon}_{\overline\Sigma}(\Gamma)$, $\epsilon>0$. Since $H^{1/2-\epsilon}_{\overline\Sigma}(\Gamma)\subseteq (\widetilde\X_{D}^{s})^{*}$ whenever $\epsilon\le 1-s$, one gets $g^{\Sigma_{\circ}}_{\lambda}\in\ran(G_{\lambda}^{\sharp}|(\widetilde\X_{D}^{s})^{*})$. 
\par 
Suppose now $\Sigma_{\circ}\cap\Sigma^{c}\not=\emptyset$ and let $D_{\circ}\subset\mathbb{R}^{n}$ open such that $\overline{D_{\circ}}\cap\Sigma=\emptyset$ and $D_{\circ}\cap \Sigma_{\circ}\neq\emptyset$. By elliptic regularity, $g^{\Sigma_{\circ}}_{\lambda}=\SL^{\circ}_{\lambda}1_{\Sigma_{\circ}}\in\C^{\infty}(\RE^{n}\backslash\overline\Sigma_{\circ})
$ (here and below, the apex ${\,}^{\circ}$ denotes objects defined by using the surface $\Gamma_{\!\circ}$) and $G^{D}_{\lambda}\phi=\SL_{\lambda}\phi\in\C^{\infty}(\RE^{n}\backslash\overline\Sigma)$ whenever $\phi\in (\widetilde\X_{D}^{s})^{*}$. By the well known jump relations for layer potentials (see, e.g., \cite[Chapter 6]{McL}) we have 
$[\gamma^{\circ}_{1}]  \SL^{\circ}_{\lambda}1_{\Sigma_{\circ}}=-1_{\Sigma_{\circ}}\neq0$; hence $g^{\Sigma_{\circ}}_{\lambda}
\notin H^{2}(  D_{\circ}) $. Assume that there exists $\phi\in (\widetilde\X_{D}^{s})^{*}$ such that%
\begin{equation}
1_{B}g^{\Sigma_{\circ}}_{\lambda}=1_{B}G^{D}_{\lambda}\phi\,.\label{cond_1}%
\end{equation}
If $D_{\circ}\subset B$ then we have $1_{B}g^{\Sigma_{\circ}}_{\lambda}\notin H^{2}(  D_{\circ})  $ and
$\SL_{\lambda}\phi\in \C^{\infty}(  D_{\circ})  $; hence the identity \eqref{cond_1} is impossible. Let then assume that $D_{\circ}\subset\mathbb{R}%
^{n}\backslash\overline{B}$. Since both $\SL^{\circ}_{\lambda}1_{\Sigma_{\circ}}$ and $\SL_{\lambda}\phi$ are radiating
solutions of the equation $(  -\Delta+\lambda)  u=0$ in $B$, \eqref{cond_1} and the unique continuation property yield $1_{B}g^{\Sigma_{\circ}}_{\lambda}=1_{B}\SL_{\lambda}\phi$ in $D_{\circ}\backslash (D_{\circ}\cap\overline\Sigma_{\circ})$, which is a contradiction.
\par\noindent
($\sharp=N$) As in the previous case, if $\Sigma_{\circ}\subset \Sigma$ then $g^{\Sigma_{\circ}}_{\lambda}=\SL_{\lambda}1_{\Sigma_{\circ}}$ with $1_{\Sigma_{\circ}}\in H^{1/2-\epsilon}_{\overline\Sigma}(\Gamma)$, $\epsilon>0$. Since $\gamma
_{1}\DL_{\lambda}\in\bou(H^{1/2}(\Gamma),H^{-1/2}(\Gamma))$ is coercive (see, e.g., \cite[Lemma 3.2]{JDE}), $R_{\Sigma}\gamma
_{1}\DL_{\lambda}R^{*}_{\Sigma}\in\bou(H_{\overline\Sigma}^{1/2}(\Gamma),H^{-1/2}(\Sigma))$ is coercive as well (see \cite[Remark 5.2]{MP-inv}) and hence $(R_{\Sigma}\gamma
_{1}\DL_{\lambda}R^{*}_{\Sigma})^{-1}\in\bou(H^{-1/2}(\Sigma),H_{\overline\Sigma}^{1/2}(\Gamma))$ (see \cite[Remark 4.6]{MP-inv}); then, by the mapping properties of $\gamma_{1}\DL_{\lambda}$ (see \cite[Theorem 3]{costa}), one gets $(R_{\Sigma}\gamma
_{1}\DL_{\lambda}R^{*}_{\Sigma})^{-1}\in\bou(H^{s-1/2}(\Sigma),H_{\overline\Sigma}^{s+1/2}(\Gamma))$, $s\in[0,1/2]$. Since both $g^{\Sigma_{\circ}}_{\lambda}$ and $\DL_{\lambda}\phi$ solve $(-\Delta+\lambda)u=0$ in $\RE^{n}\backslash\overline\Sigma_{\circ}$ 
with boundary condition $R_{\Sigma}\gamma_{1}u=R_{\Sigma}\gamma_{1}g^{\Sigma_{\circ}}_{\lambda}$ whenever $\phi=(R_{\Sigma}\gamma_{1}\DL_{\lambda}R^{*}_{\Sigma})^{-1}R_{\Sigma}\gamma_{1}g^{\Sigma_{\circ}}_{\lambda}$ and such a problem has a unique  radiating solution (see \cite[Theorem 3.3]{Agra}),  one gets $g^{\Sigma_{\circ}}_{\lambda}=\DL_{\lambda}\phi=G^{N}_{\lambda}\phi$, $\phi\in H_{\overline\Sigma}^{s+1/2}(\Gamma)$.\par
If $\Sigma_{\circ}\cap\Sigma^{c}\not=\emptyset$ one shows that $1_{B}g^{\Sigma_{\circ}}_{\lambda}\notin \ran(1_{B}G_{\lambda}^{N}|(\widetilde\X_{N}^{s})^{*})$ by the same kind of reasonings used in the case $\sharp=D$.\end{proof}


\begin{thebibliography}{99}
\bibitem{Agra} M. S. Agranovich: Strongly elliptic second order systems with spectral parameter in transmission conditions on a nonclosed surface. In: P. Boggiatto, L. Rodino, J. Toft, M. W. Wang (eds.), {\it Pseudo-differential operators and related topics}, Operator Theory:
Advances and Applications, vol. 164, 1-21, Birkh\"auser, Basel, 2006. 
\bibitem{A} W. Arendt, C.J.K. Batty, M. Hieber, F. Neubrander:
{\it Vector-valued
Laplace Transforms and Cauchy Problems.
2nd ed.}  Birkh\"auser 2011
\bibitem{borcea} L. Borcea, G. Papanicolaou, C. Tsogka:
Adaptive interferometric imaging in clutter and optimal illumination. 
{\it Inverse Problems} {\bf 22} (2006), 1405-1436. 
\bibitem{CFP} C. Cacciapuoti, D. Fermi, A. Posilicano: On inverses of Kre\u\i n's $\mathscr Q$-functions, {\it Rend. Mat. Appl.} {\bf 39} (2018), 229-240.
\bibitem {CaHaLe} F. Cakoni, H. Haddar, A. Lechleiter: On the Factorization
Method for a Far Field Inverse Scattering Problem in the Time Domain. 2018. hal-0194566.
\bibitem {ChHaLeMo} Q. Chen, H. Haddar, A. Lechleiter, P. Monk: A sampling
method for inverse scattering in the time domain. \emph{Inverse Problems},
\textbf{26} (2010), 085001.
\bibitem{costa} M. Costabel: Boundary Integral Operators on Lipschitz Domains: Elementary Results. { SIAM J. Math. Anal.} {19} (1988), 613-626. 
\bibitem{Fatto} H.O. Fattorini: {\it 
Second order linear differential equations in Banach spaces.} North-Holland Publishing Co. 1985.
\bibitem {HaLeMa} H. Haddar, A. Lechleiter, S. Marmorat: An improved time
domain linear sampling method for Robin and Neumann obstacles. \emph{Appl.
Anal.} \textbf{93} (2014), 369-390.
\bibitem{Isakov} V. Isakov: Inverse obstacle problems. {\it Inverse Problems}, {\bf  25} (2009), 123002 (18pp).
\bibitem{Jais} M. Jais: Parameter identification for Maxwell's equations. PhD Thesis, Cardiff University, 2006 (http://orca.cf.ac.uk/54581/).
\bibitem{JW} A. Jonsson, H. Wallin: Function Spaces on Subsets of $\RE^n$.
{\it Math. Reports}, {\bf 2} (1984), 1-221.
\bibitem{Jorg} K. J\"orgens: {\it Linear Integral Operators.} Pitman, London, 1982.
\bibitem{KG} A. Kirsch, N. Grinberg: {\it The Factorization Method for Inverse Problems.} Oxford University Press, Oxford, 2008.
\bibitem {LuPo} D.R. Luke, R. Potthast: The point source method for inverse
scattering in the time domain. \emph{Math. Meth. Appl. Sci.} \textbf{29} (2006),
1501-1521.
\bibitem{McL} W. McLean: \emph{Strongly elliptic systems and boundary
\ integral equations.} Cambridge University Press, Cambrdige, 2000.
\bibitem{JMPA} A. Mantile, A. Posilicano: Asymptotic Completeness and S-Matrix for Singular 
Perturbations.  {\it J. Math. Pures Appl.} {\bf 130} (2019), 36-67.
\bibitem{MP-inv} A. Mantile, A. Posilicano: Inverse Scattering for the Laplace Operator with Boundary Conditions on Lipschitz Surfaces. {\it Inverse Problems} {\bf 35} (2019), 124007 (27pp).
\bibitem{JDE} A. Mantile, A. Posilicano, M. Sini: Self-adjoint elliptic operators with boundary conditions on not closed hypersurfaces. {\it J. Differential Equations} {\bf 261} (2016), 1-55.
\bibitem{JST} A. Mantile, A. Posilicano, M. Sini: Limiting Absorption Principle, Generalized Eigenfunctions and Scattering Matrix for Laplace Operators with Boundary conditions on Hypersurfaces. {\it J. Spectr. Theory} {\bf 8} (2018), 1443-1486.
\bibitem {P01}A. Posilicano: A Kre\u{\i}n-like formula for singular
perturbations of self-adjoint operators and applications. \emph{J. Funct.
Anal.} \textbf{183} (2001), 109-147.

\end{thebibliography}
\end{document}